\documentclass[12pt,reqno,a4paper]{amsart}
\usepackage{amsmath,amssymb,amsthm,amsaddr}
\usepackage{enumitem}
\usepackage[margin=2.7cm,top=3.5cm,bottom=3.5cm,footskip=1cm,headsep=1.5cm]{geometry}
\usepackage[utf8]{inputenc}

\usepackage{amsthm}
\usepackage[british]{babel}

\author{H. Egger}
\address{Department of Mathematics, TU Darmstadt, Germany}
\email{egger@mathematik.tu-darmstadt.de}

\title[Approximation of dissipative evolution problems]{Structure preserving approximation\\[0.5ex]of dissipative evolution problems}

\newtheorem{theorem}{Theorem}

\theoremstyle{definition}

\def\A{\mathcal{A}}
\def\D{\mathcal{D}}
\def\E{\mathcal{E}}
\def\Q{\mathcal{Q}}

\def\Ttau{T_\tau}

\def\HH{\mathbb{H}}
\def\RR{\mathbb{R}}
\def\VV{\mathbb{V}}
\def\WW{\mathbb{W}}

\def\dt{\partial_t}
\def\dx{\partial_x}
\def\dn{\partial_n}
\def\ddt{\frac{d}{dt}}
\def\dtau{\partial_\tau}
\def\div{\mathrm{div}}
\def\curl{\mathrm{curl}}
\def\eps{\epsilon}
\def\dom{\mathrm{dom}}

\numberwithin{equation}{section}

\begin{document}


\begin{abstract}
We present a general abstract framework for the systematic numerical approximation of dissipative evolution problems. The approach is based on rewriting the evolution problem in a particular form that complies with an underlying energy or entropy structure. Based on the variational characterization of smooth solutions, we are then able to show that the approximation by Galerkin methods in space and discontinuous Galerkin methods in time automatically leads to numerical schemes that inherit the dissipative behavior of the evolution problem. The proposed framework is rather general and can be applied to a wide range of applications. This is demonstrated by a detailed discussion of a variety examples ranging from diffusive partial differential equations to Hamiltonian and gradient systems. 
\end{abstract}

\maketitle

\vspace*{-1em}

\begin{quote}
\end{quote}

\begin{quote}
\noindent
{\small {\bf AMS-classification (2000):}
37K05, 
37L65, 
47J35, 
65J08  
}
\end{quote}

\section{Introduction} \label{sec:1}

The scope of this paper is to devise a general framework for the systematic construction of numerical approximation schemes for dissipative evolution problems that are accompanied by an energy or entropy structure. 
Such problems have been studied intensively in the literature over the last years, in particular, in connection with the analysis and numerical approximation of nonlinear partial differential equations. Let us refer to \cite{Evans08,Roubicek13}  and \cite{Evans04,Juengel16,Matthes07} for an introduction to the field and further references.

\subsection*{Outline of the approach.}
Our starting point and basic assumption is that the evolution problem under consideration can be stated in the abstract form 
\begin{align} \label{eq:11}
\Q(u)^* \dt u = \A(u), \qquad t>0,
\end{align}
that complies with an associated energy functional $\E(\cdot)$ in the sense that $\Q(u)^*$ is the adjoint of the operator $\Q(u)$ which is related to the derivative of this functional by
\begin{align} \label{eq:12}
\E'(u) = Q(u) u. 
\end{align}
Based on this simple structural assumption, one can verify that any smooth solution of the evolution problem \eqref{eq:11} satisfies a \emph{dissipation identity} of the form
\begin{align} \label{eq:13}
\ddt \E(u(t)) = \langle \A(u),u\rangle =: -\D(u(t)).
\end{align}
Hence, the functional $\D(\cdot)$ describes the rate at which energy is dissipated. 
The dynamical system under consideration is called \emph{dissipative}, if $\D(u) \ge 0$. Note that conservative systems are included as a special case with $\D(u)=0$.
As already mentioned above, such dissipation identities are of great importance for the analysis and numerical approximation of such systems by energy or entropy methods; see e.g. \cite{Evans08,Roubicek13} and \cite{Evans04,Juengel16,Matthes07}.

\subsection*{Overview about results.}
It is clear that an evolution problem can be written in many, at least formally, equivalent ways. 
The particular form stated above, however, has the following important advantages for the numerical approximation. 
\begin{itemize}\itemsep1ex
 \item The dissipation identity \eqref{eq:13} here follows by simply testing the variational form 
       \begin{align} \label{eq:14}
        \langle Q(u(t))^* \dt u(t), v\rangle &= \langle \A(u(t)),v\rangle, \qquad \forall v \in \VV, \ t>0,
       \end{align}
       of the evolution equation \eqref{eq:11} with the test function $v=u(t)$ and using the structural relation \eqref{eq:12} between $\Q(u)$ and $\E'(u)$. 
       Let us note that our arguments are therefore naturally associated with a weak solution concept.
 \item For approximations $u_h(t)$ obtained by Galerkin projection of this variational principle to a subspace $\VV_h \subset \VV$, a corresponding discrete dissipation identity 
       \begin{align} \label{eq:15}
        \ddt \E(u_h(t)) =  -\D(u_h(t)), \qquad  t>0,
       \end{align}
       is valid automatically, which can be proven with the same arguments as on the continuous level. 
       The geometric structure of the problem is thus inherited. 
 \item The approximations $(u^n)_{n \ge 0}$ obtained by the implicit Euler method applied to the variational principle above satisfy a discrete dissipation inequality
       \begin{align} \label{eq:16}
       \dtau \E(u^n) \le -\D(u^n), \qquad n > 0, 
       \end{align}
       where $\dtau E(u^n)$ denotes the backward difference quotient in time. Again, the underlying dissipation structure is inherited automatically. Note that due to numerical dissipation, an inequality is obtained here instead of an equality. We will further show that discretization in time by discontinuous Galerkin methods allows to obtain similar results also for approximations of higher order. 
\end{itemize}

\subsection*{Summary.} 
A structure preserving numerical approximation of dissipative evolution problems can be achieved 
in a systematic manner, if the problem is stated in the appropriate form already on the continuous level.
We will illustrate by several examples that this is possible for a wide range of applications.

\subsection*{Previous results.}
Before we proceed, let us briefly discuss some related literature:
Energy conservation or dissipation or entropy production and the preservation of these properties during numerical approximation of evolution problems is of great interest already from an analytical point of view, e.g., to obtain uniform a-priori bounds for numerical approximations used for establishing existence of solutions to nonlinear evolution problems; see  \cite{Evans08,Roubicek13}  and \cite{Evans04,Juengel16,Matthes07} for examples and further references.

The use of energy estimates for the numerical analysis of Galerkin approximations is well-developed for simple evolution problems; see e.g. \cite{Thomee06} and the references given there.
In contrast to that, the design and analysis of structure preserving or dissipative discretization schemes for nonlinear evolution problems still seems at an early stage of research, and only partial results are available for specific problems; see \cite[Ch.~5]{Juengel16} for a recent review and further references.

Let us briefly mention some particular results:
One-leg multistep methods and implicit Runge-Kutta methods have been investigated for the time discretization of dissipative evolution problems in \cite{JuengelMilisic15,JuengelSchuchnigg17}.  
Apart from the implicit Euler method, however, the assumptions required for the rigorous analysis of these schemes seem rather restrictive. 
Dissipative finite volume methods for the Fokker-Planck equation have been analyzed in \cite{Mielke13} and mixed finite element approximations for nonlinear diffusion problems were investigated in \cite{BurgerCarilloWolfram10}. 
Further examples for entropy based finite element approximations of particular applications are \cite{BarrettBlowey98}, which is concerned with degenerate Allen-Cahn and Cahn-Hillard models, and \cite{ProhlSchmuck09}, dealing with a Nernst-Planck-Poisson system. 
In a similar spirit, a discontinuous Hamiltonian finite element method for the approximation of linear hyperbolic systems was proposed in \cite{XuVanDerVegtBokhove08}.

The philosophy of the current manuscript seems to differ substantially from these previous approaches: 
instead of developing special approximation schemes for individual problems, we here attempt to provide a unified framework that can be applied to a wide range of applications and which leads to dissipative discretization methods automatically. 
We strongly believe that this approach may be useful for many applications and it may serve as a starting point for the analysis, the proof of convergence and convergence rates, et cetera. These are left as topics for future research.

\subsection*{Outline.}
The remainder of the manuscript is organized as follows:
In Section~\ref{sec:2}, we present in more detail the general framework and the basic assumptions characterizing the dissipative structure of the underlying evolution problem. 
In Section~\ref{sec:3}, we discuss the systematic approximation by Galerkin projection in state space and prove the discrete dissipation inequality for the discontinuous Galerkin discretization in time. 
Sections~\ref{sec:4}-\ref{sec:10} are devoted to the discussion of several examples taken from literature. As we will see, our approach is applicable to all test problems and provides a recipe for the systematic construction of numerical approximation schemes. 
Some aspects that would deserve further investigation will be highlighted at the end of the manuscript.

\section*{Part 1: The general framework}

In the following two sections, we first introduce our basic assumptions and the problems to be considered and then discuss their systematic numerical approximation. 

\section{Problem setting} \label{sec:2}

Let us start with discussing the general abstract form of evolution problems that are compatible with a governing dissipation structure.
The presentation will be somewhat formal, i.e., we do not try to be as general or rigorous as possible, but instead, we choose a functional analytic setting that is simple enough to allow for a convenient presentation and at the same time general enough to capture the main aspects. 

\subsection{Function spaces}

Let $\HH$ be a real Hilbert space with scalar product $\langle \cdot, \cdot\rangle$. 
We identify $\HH$ with its dual space $\HH^*$ and the scalar product on $\HH$ with the duality product on $\HH^* \times \HH$. 
Let $\VV,\WW \subset \HH$ be two reflexive Banach spaces which are continuously and densly embedded in $\HH$.
Note that, by embedding and the identification of $\HH$ with $\HH^*$, we can interpret $\HH$ as a dense subspace of the dual spaces $\VV^*$ and $\WW^*$ and thus obtain two evolution triples $\VV \subset \HH \subset \VV^*$ and $\WW \subset \HH \subset \WW^*$. Since both triples are based on the same pivot space $\HH$, we also obtain the natural inclusions $\VV \subset \WW^*$ and $\WW \subset \VV^*$.
We refer to \cite{Roubicek13} for details on the notation and further information.

\subsection{Energy functional}

Let $\E : \VV \subset \WW^*  \to \RR$ be a given energy functional that is 
assumed to be convex, proper, and differentiable on its domain with respect to the topology of $\WW^*$.
Then by reflexivity of the space $\WW$, the derivative $\E'(u) \in \WW^{**}$ can be understood as an element of $\WW$. 
The main structural assumption for our approach is that the derivative of the energy functional can further be represented in the form
\begin{align} \label{eq:21}
\E'(u) = \Q(u) u, \qquad \text{for all } u \in \dom(\E) \subset \VV,
\end{align}
with bounded linear operators $\Q(u) : \VV \to \WW$. 
Here $\dom(\E)$ is the essential domain of the functional $\E$, 
i.e., the set of all $u$ such that $\E(u)$ is finite.
By the identities
\begin{align} \label{eq:22}
\langle \Q(u)^* w^*, v\rangle_{\VV^* \times \VV} = \langle w^*, \Q(u) v\rangle_{\WW^* \times \WW}, \qquad \forall v \in \VV, \ w^* \in \WW^*, 
\end{align}
we define corresponding adjoint operators $\Q(u)^* : \WW^* \to \VV^*$, again linear and bounded.

\subsection{Evolution problem}

In order to comply with the underlying energy dissipation structure, we require that the evolution problem is 
given in the abstract form 
\begin{align}  \label{eq:23}
\Q^*(u) \dt u  = \A(u), \qquad \text{for all } t > 0,
\end{align}
where $\A : \VV \to \VV^*$ is some suitable densly defined operator. We denote by 
\begin{align} \label{eq:24}
-\D(u) := \langle \A(u), u\rangle_{\VV^* \times \VV}, \qquad \forall u \in \dom(\A),
\end{align}
the associated dissipation functional $\D : \VV \to \RR$ which describes the rate at which energy is dissipated.
In most cases of practical interest, $\D(u)$ will be non-negative.

\subsection{Structure theorem}

Under the above assumptions, any smooth solution of the abstract evolution problem \eqref{eq:23} satisfies the following dissipation identity.

\begin{theorem} \label{thm:1}
Let $u : [0,T] \to \VV \subset \WW^*$ be a smooth solution of \eqref{eq:23}, i.e., $u$ is continuous in $t$ with respect to the norm of $\VV$ and continuously differentiable with respect to the norm of $\WW^*$; moreover, $u(t) \in \dom(\A)$ for all $t$ and $u(0)\in\dom(\E)$. 
Then 
\begin{align} \label{eq:25}
\frac{d}{dt} \E(u(t)) = -\D(u(t)) \qquad  \text{for all } t>0. 
\end{align}
\end{theorem}
\begin{proof}
Formal differentiation of $\E(u(t))$ with respect to time yields
\begin{align*}
\frac{d}{dt} \E(u(t)) 
&= \langle \dt u(t), \E'(u(t))\rangle_{\WW^* \times \WW}   
 = \langle \dt u(t), \Q(u(t)) u(t) \rangle_{\WW^* \times \WW} \\
&= \langle \Q^*(u(t)) \dt u(t), u(t) \rangle_{\VV^* \times \VV} 
 = \langle \A(u(t)), u(t)\rangle_{\VV^* \times \VV} = -\D(u(t)).
\end{align*}
A quick inspection of the individual steps reveals that all terms are well-defined under the regularity assumptions on the solution and the energy functional made above.
\end{proof}

\subsection{Remarks}
By integration in time, one can also obtain an integral form 
\begin{align} \label{eq:26}
\E(u(t)) = \E(u(s)) - \int_s^t \D(u(r)) dr, \qquad 0 < s \le t, 
\end{align}
of the dissipation identity, which again holds for all smooth solutions of problem \eqref{eq:23}. 
For generalized solutions that are obtained as limits of certain approximations, one would rather expect a corresponding dissipation inequality; see e.g. \cite{Feireisl03} for details.   
%

%

\section{Structure preserving discretization} \label{sec:3}

Let us note that any sufficiently smooth solution $u :  [0,T] \to \VV$ of the evolution problem \eqref{eq:23} can be characterized by the variational principle 
\begin{align} \label{eq:31}
\langle \Q(u(t))^* \dt u(t), v\rangle_{\VV^* \times \VV} = \langle \A(u(t)), v\rangle_{\VV^* \times \VV} \qquad v \in \VV, t > 0, 
\end{align}
which is equivalent to \eqref{eq:23}, but which can also be used for defining an appropriate weak solution concept. 
As we will illustrate now, this variational characterization or corresponding weak formulations are appropriate for the numerical approximation by Galerkin projection in space and a discontinuous Galerkin approximation in time. 
For both discretization approaches, a discrete dissipation identity or inequality can be derived under general assumptions and with relatively simple arguments. 

\subsection{Galerkin approximation in space}
Let $\VV_h \subset \VV$ denote some closed subspace of the state space $\VV$. For the semi-discretization of the evolution problem \eqref{eq:23} in space, we consider the following discrete variational principle
\begin{align} \label{eq:32}
\langle \Q(u_h(t))^* \dt u_h(t), v_h\rangle_{\VV^* \times \VV} 
= \langle \A(u_h(t)), v_h\rangle_{\VV^* \times \VV}, \qquad \forall v_h \in \VV_h, \ t > 0. 
\end{align}
Appropriate initial conditions are, of course, required to determine the numerical solution uniquely. 
Due to the specific structure of the underlying evolution problem, the dissipation identity is inherited automatically by the Galerkin approximation.
\begin{theorem} \label{thm:2}
Let $u_h : [0,T] \to \VV_h$ denote a smooth solution of \eqref{eq:32}.
Then
\begin{align} \label{eq:33} 
\frac{d}{dt} \E(u_h(t)) = -\D(u_h(t)) \qquad \text{for all } t > 0.
\end{align}
\end{theorem}
\begin{proof}
The proof of Theorem~\ref{thm:1} applies verbatim. 
\end{proof}

\subsection*{Remark.}
Let us emphasize the generality of this result which formally covers any evolution problem of the required form and any Galerkin approximation thereof.

\subsection{Time discretization}

As a second discretization step, we consider the numerical approximation in time. Let $\Ttau=\{0=t^0 < t^1 < t^2 < \ldots\}$ be an increasing sequence of time points and let $P_k([t^{n-1},t^n];\VV)=\{u : u=a_0 + a_1 t + \ldots a_k t^k, \ a_j \in \VV\}$ be the space of polynomials on $[t^{n-1},t^n]$ with values in $\VV$. We further denote by
\begin{align} \label{eq:34}
P_k(\Ttau;\VV) = \left\{ u : u^n := u|_{[t^{n-1},t^n]}  \in P_k([t^{n-1},t^n];\VV)\right\} 
\end{align}
the space of piecewise polynomial functions of time with values in $\VV$. Note that functions in $P_k(\Ttau;\VV)$ are smooth on every interval $[t^{n-1},t^n]$, but they may in general be discontinuous at the time points $t^n$, $n>0$, between two intervals. In this case, they have two different values at $t^n$, $n>0$, defined as the limits from above and below.

The discontinuous Galerkin discretization of the variational principle \eqref{eq:31} in time characterizes approximations $u \in P_k(\Ttau;\VV)$ by the discrete variational principle
\begin{align} \label{eq:35}
&\int_{t^{n-1}}^{t^n} \langle \Q(u^n(t))^* \dt u^n(t), v\rangle_{\VV^*\times\VV} dt 
= \int_{t^{n-1}}^{t^n} \langle \A(u^n(t)), v\rangle_{\VV^*\times\VV} dt\\
& \qquad   - \langle \Q(u^n(t^{n-1}))^* (u^n(t^{n-1}) - u^{n-1}(t^{n-1})), v\rangle_{\VV^*\times\VV}, \qquad \forall v \in P_k([t^{n-1},t^n];\VV), \ n>0. \notag
\end{align}
Using the convexity of the energy functional $\E(\cdot)$, the dissipation structure of the evolution problem, and the dissipative nature of the discontinuous Galerkin method, 
we are able to establish the following general dissipation inequality.
\begin{theorem} \label{thm:3}
Let $u \in P_k(\Ttau;\VV)$ denote a solution of the scheme \eqref{eq:35}. Then 
\begin{align} \label{eq:36}
\E(u^n(t^n)) \le \E(u^m(t^m)) - \int_{t^m}^{t^n} \D(u(t)) dt, \qquad 0 \le m < n.
\end{align}
{\em This corresponds to a discrete version of the dissipation identity \eqref{eq:26} in integral form. Due to numerical dissipation, an inequality is obtained here instead of an identity.} 
\end{theorem}
\begin{proof}
By basic manipulations and the fundamental theorem of calculus, we obtain 
\begin{align*}
\E(u^n(t^n)) &- \E(u^{n-1}(t^{n-1})) \\
&= \E(u^n(t^n)) - \E(u^n(t^{n-1})) + \E(u^n(t^{n-1})) - \E(u^{n-1}(t^{n-1})) \\
&= \int_{t^{n-1}}^{t^n} \frac{d}{dt} \E(u^n(t)) dt  + \E(u^n(t^{n-1})) - \E(u^{n-1}(t^{n-1})) = (i) + (ii).
\end{align*}
By means of the structure relation \eqref{eq:21}, the integrand can be written as 
\begin{align*}
\frac{d}{dt} \E(u^n(t)) 
&= \langle \dt u^n(t), \E'(u^n(t))\rangle_{\WW^* \times \WW} \\
&= \langle \dt u^n(t), Q(u^n(t)) u^n(t))\rangle_{\WW^* \times \WW} 
= \langle Q(u^n(t))^* \dt u^n(t), u^n(t))\rangle_{\VV^* \times \VV}. 
\end{align*}
Integration with respect to time and using equation \eqref{eq:35} with $v=u^n$ then yields
\begin{align*}
(i) 
&= \int_{t^{n-1}}^{t^n} \langle \A(u^n(t)), u^n(t) \rangle_{\VV^* \times \VV} dt \\
&\qquad  - \langle Q(u^n(t^{n-1}))^* (u^n(t^{n-1}) - u^{n-1}(t^{n-1})), u^n(t^{n-1})\rangle_{\VV^* \times \VV} = (iii) + (iv).  
\end{align*}
By identity \eqref{eq:24}, the term $\langle \A(u),u\rangle$ in (iii) can simply be replaced by $\D(u)$.
The remaining terms (ii) and (iv) in the above estimates can be treated as follows: 
For ease of notation, let us define $a=u^n(t^{n-1})$ and $b=u^{n-1}(t^{n-1})$. Then 
\begin{align*}
(ii) + (iv) 
&= \E(a) - \E(b) - \langle Q(a)^* (a-b), a\rangle  \\
&= \E(a) - \E(b) - \langle a-b, \E'(a) \rangle \le 0,
\end{align*}
where we used the structure relation \eqref{eq:21} for the second identity and the convexity of the energy functional $\E(\cdot)$ for the last inequality. 
This already proves the assertion of the theorem for $m=n-1$. 
The case $m<n-1$ simply follows by induction.
\end{proof}

\subsection*{Remark.}
For polynomial degree $k=0$, the sequence $(u^n)_{n \ge 0}$ obtained by the discontinuous Galerkin method coincides with the iterates generated by the implicit Euler method. The discrete dissipation inequality announced in the introduction then follows from that of Theorem~\ref{thm:3} by setting $m=n-1$ and rearranging the terms.

\subsection*{Remark.}
Since the underlying dissipation structure is preserved by Galerkin approximation in space, the above time discretization strategy can also be applied to the Galerkin semi-discretization of the underlying evolution problem. This allows to obtain energy dissipative fully discrete approximation schemes.

\section*{Part II: Diffusion problems}

We now demonstrate the general applicability of our approach by a variety of typical test examples. The first set of problems is concerned with diffusive partial differential equations. Due to the physical background, the term \emph{entropy} is often used in the literature instead of \emph{energy} as we do here. Related analytical and numerical results can therefore be found under the name \emph{entropy methods}; see e.g. \cite{Evans04,Juengel16}. 

\section{Heat equation} \label{sec:4}

One of the simplest models for diffusion processes is given by
the linear heat equation 
\begin{alignat*}{2} 
\dt u &= \Delta u, \qquad && x \in \Omega, \ t>0, \\
    0 &= \dn u,    \qquad && x \in \partial\Omega, \ t>0.
\end{alignat*}
Instead of a quadratic energy functional that is usually employed \cite{Evans08,Thomee06}, we here consider as in \cite{Evans04} the negative logarithmic entropy as an energy functional, i.e., 
\begin{align*}
\E(u) = -\int_\Omega \log u \; dx.
\end{align*}
The derivative of this energy functional can be expressed as
\begin{align*}
\langle \E'(u),v\rangle = -\langle u^{-1}, v\rangle = -\langle u^{-2} u, v\rangle,
\end{align*}
where we used $\langle u,v\rangle=\int_\Omega u v dx$ to abbreviate the scalar product of $L^2(\Omega)$. 
The derivative can thus be decomposed in the form $\E'(u)=\Q(u) u$ with operators $\Q(u)$ and their adjoints $\Q(u)^*$ that can formally be identified with the multiplication operators 
\begin{align*}
\Q(u) v =-u^{-2} v \qquad \text{and} \qquad \Q(u)^* v = -u^{-2} v. 
\end{align*}
The abstract framework presented in Section~\ref{sec:2} now suggests that, instead of the linear heat equation, we should rather consider the equivalent nonlinear equation 
\begin{align} \label{eq:8}
-\frac{1}{u^2} \dt u &= -\frac{1}{u^2} \Delta u, \qquad  x \in \Omega, \ t>0,
\end{align}
in order to comply with the dissipation structure induced by the logarithmic energy functional above. 
The corresponding operator $\A(u)$ for this problem is then given by 
\begin{align*}
\langle \A(u),v\rangle
=-\langle u^{-2} \Delta u, v\rangle
= -\langle u \nabla (u^{-1}), u \nabla (u^{-2} v)\rangle.
\end{align*}
The second identity, which follows from integration-by-parts, use of the boundary conditions, and some elementary computations, provides a weak form of the operator $\A(u)$. 
From this weak representation, one can immediately deduce that 
\begin{align*}
-\D(u) := \langle \A(u),u\rangle = -\|u \nabla (u^{-1})\|^2_{L^2(\Omega)} \le 0.
\end{align*}
From the abstract result of Theorem~\ref{thm:1}, we deduce that $\ddt \E(u(t)) \le -\D(u(t)) \le 0$, i.e., the above logarithmic energy of the system is decreasing or, equivalently, the entropy is increasing along the evolution of the dynamical system.

By the results of Section~\ref{sec:3}, 
a structure-preserving numerical approximation can now be realized as follows: We can use a standard Galerkin approximation of the nonlinear variational principle \eqref{eq:22} with continuous and piecewise linear finite elements in space and an implicit Euler method in time. By Theorem~\ref{thm:2} and \ref{thm:3}, 
the resulting discrete approximations automatically inherit the underlying dissipation structure,
i.e., the logarithmic energy will be monotonically decreasing for the numerical solutions. 
\subsection*{Remark.}
The resulting discretization scheme is based on the nonlinear differential equation $u^{-2} \dt u = u^{-2} \Delta u$ and can be interpreted as a nonlinear approximation scheme for the linear heat equation. Note that according to \eqref{eq:11} and \eqref{eq:12}, the form of the approximation scheme is already determined by the underlying energy functional.

\section{Porous medium equation} \label{sec:5}

We next turn to nonlinear diffusion processes. 
Let $\Omega \subset \RR^d$, $d \ge 1$ be some bounded Lipschitz domain and choose $m>1$. We consider the porous medium equation
\begin{alignat*}{3}
\dt \rho &= \Delta \rho^m  \qquad && \text{in } \Omega, \\
0 &= \dn \rho^{m} && \text{on } \partial\Omega.
\end{alignat*}
A natural candidate for an energy suitable for the analysis of this problem is \begin{align*}
\E(\rho) = \int_\Omega \tfrac{1}{m-1}  \rho^m dx.
\end{align*}
We refer to \cite{Vazquez07} for a complete treatment of the problem based on entropy arguments.
The derivative of the above energy functional is given by 
\begin{align*}
\langle \E'(\rho), v \rangle 
= \int_\Omega \tfrac{m}{m-1} \rho^{m-1} v dx 
= \int_\Omega \tfrac{m}{m-1} \rho^{m-2} \rho v dx.
\end{align*}
One can see that the derivative can be decomposed in the form $\E'(\rho)=\Q(\rho)\rho$ with operator $\Q(\rho)$ and its adjoint $\Q(\rho)^*$ formally defined by 
\begin{align*}
\Q(\rho) u = \tfrac{m}{m-1} \rho^{m-2} u
\qquad \text{and} \qquad 
\Q(\rho)^* v = \tfrac{m}{m-1} \rho^{m-2} v.
\end{align*}
With some abuse of notation, we again identified the operators $\Q(u)$ and $\Q(u)^*$ with the corresponding multiplication operators. 
Following the general framework developed in Section~\ref{sec:2}, 
we now rewrite the porous medium equation in the non-conventional form
\begin{align*}
\tfrac{m}{m-1} \rho^{m-2} \dt \rho 
&= \tfrac{m}{m-1} \rho^{m-2} \Delta \rho^m  \\
&= \tfrac{m}{m-1} \rho^{m-2} \div \left(\tfrac{m}{m-1} \rho \nabla \rho^{m-1}\right).
\end{align*}
Multiplying with a test function $v$, integrating over the domain $\Omega$, 
using integration-by-parts, and the boundary conditions 
here leads to the weak formulation 
\begin{align} \label{eq:41}
\langle \Q(\rho)^* \dt \rho, v\rangle
&=\left(\tfrac{m}{m-1} \rho^{m-2} \dt \rho, v \right)_\Omega \\
&=-\left(\rho \tfrac{m}{m-1} \nabla \rho^{m-1}, \tfrac{m}{m-1} \nabla (\rho^{m-2} v) \right)_\Omega 
 =: \langle \A(\rho), v\rangle, \notag
\end{align}
which is assumed to hold for all suitable test functions $v$ and all $t>0$. 
Let us note that the solution $\rho=\rho(t)$ depends on time $t$ while the test function $v$ does not.
It is not difficult to see that the operator $\A(\cdot)$ is dissipative in the sense that
\begin{align*}
-\D(\rho):=\langle \A(\rho), \rho \rangle = -\int_\Omega \rho \left|\tfrac{m}{m-1} \nabla \rho^{m-1}\right|^2 dx \le 0,
\end{align*}
whenever the density $\rho \ge 0$ stays non-negative; this can be guaranteed by comparison principles \cite{Vazquez07}. 
Assuming the non-negativity of the solution, we thus obtain 
\begin{align*}
\frac{d}{dt} \int_\Omega \tfrac{1}{m-1} \rho(t)^m dx 
&= - \int_\Omega \rho(t) \left|\tfrac{m}{m-1} \nabla \rho(t)^{m-1} \right|^2 dx 
 = - \int_\Omega \left| \tfrac{2m}{2m-1} \nabla \rho(t)^{\frac{2m-1}{2}}\right|^2 dx,
\end{align*}
which is exactly the dissipation identity $\ddt \E(\rho) = - \D(\rho) \le 0$ provided by Theorem~\ref{thm:1}. 
As a direct consequence, one can see that the $L^p$-norm of the solution is uniformly bounded 
if the initial values are bounded appropriately. 
By integration in time, one can additionally obtain uniform bounds for the spatial derivatives. This is the starting point for establishing existence of solutions; we refer to \cite{Vazquez07} for details.

For discretization of the problem, we can now simply use a Galerkin approximation of the variational principle \eqref{eq:41} by piecewise linear finite elements combined with an implicit Euler method in time. As a consequence of Theorem~\ref{thm:2} and \ref{thm:3}, the fully discrete solution will automatically satisfy the dissipation inequality $\dtau \E(u_h^n) \le - \D(u_h^n) \le 0$, which is of a similar form as the dissipation identity of the continuous solution and implies uniform a-priori bounds for the discrete approximations. 

\section{Fokker-Planck equation} \label{sec:6}

Another class of problems that have been studied intensively in the context of entropy methods are Fokker-Planck equations.  
We here consider the linear problem
\begin{alignat*}{2}
\dt \rho &= \div (\nabla \rho + \rho \nabla V), \qquad &&\text{in } \Omega, \ t>0, \\
    0 &= \dn \rho + \rho \dn V                      &&\text{on } \partial\Omega, \ t>0,
\end{alignat*}
where $\rho$ is an unknown density to be determined and $V : \Omega \to \RR$ is a prescribed potential. 
Following \cite{CarrilloEtAl01,Juengel16}, we define $u=\rho/\rho_\infty$, with $\rho_\infty(x)= c e^{-V(x)}$ denoting a 
solution of the corresponding stationary problem. Since the equation is in conservative form, the constant $c$ should be chosen such that $\int_\Omega \rho_\infty dx = \int_\Omega \rho(0) dx$. 
Using the new variable $u$, the above problem can be rewritten as 
\begin{alignat*}{2}
\rho_\infty \dt u &= \div (\rho_\infty \nabla u), \qquad && \text{in } \Omega, \ t>0,\\
                0 &= \rho_\infty \dn u, \qquad && \text{on } \partial\Omega, \ t>0. 
\end{alignat*}
Note that $\rho_\infty$ can be assumed to be positive, independent of time, and known a-priori.
As an energy governing the evolution, we here utilize the quadratic functional
\begin{align*}
\E(u) = \int_\Omega \tfrac{1}{2} u^2 \rho_\infty dx;
\end{align*}
see e.g. \cite[Ch.~2]{Juengel16}.
The derivative of this energy is given by 
$\langle \E'(u),v\rangle = \langle \rho_\infty u,v\rangle$
and can be decomposed as $\E'(u) = \Q(u) u$ with 
$\Q(u)$ and adjoint $\Q(u)^*$ defined by
\begin{align*}
\Q(u) v = \rho_\infty v \qquad \text{and} \qquad \Q(u)^* v = \rho_\infty v. 
\end{align*}
We again identified the operators $\Q(u)$ and $\Q(u)^*$ with the corresponding multiplication operators. 
With these definitions, one can see that the above equation for $u$ is already 
in the appropriate form $\Q(u)^* \dt u = \A(u)$ required for our framework. 
The corresponding weak formulation of the problem here reads 
\begin{align} \label{eq:ex2var}
\langle \Q(u)^* \dt u,v\rangle
&= \langle \rho_\infty \dt u, v\rangle   
 = -\langle \rho_\infty \nabla u, \nabla v\rangle =: \langle \A(u),v\rangle. 
\end{align}
By testing this variational principle with $v=u$, we can extract the dissipation functional 
\begin{align*}
-\D(u) 
:= \langle \A(u), u \rangle 
&= -\int_\Omega \rho_\infty |\nabla u|^2 dx. 
\end{align*}
The above derivations and Theorem~\ref{thm:1} show that the energy $\E(u)$ will be monotonically decreasing, unless $u \equiv c_1$ constant. Based on the conservation of $\rho$ resulting from the first formulation of the problem, one can see that $c_1=1$ must hold in that case. 
One can even show that convergence to the steady state takes place exponentially fast \cite{CarrilloEtAl01,Juengel16}. 

For the discretization of the variational principle \eqref{eq:ex2var}, we can again use a standard finite element approximation in space and a discontinuous Galerkin method in time. 
This will lead to a numerical approximation with the same dissipation behavior as the continuous problem and which can be expected to converge exponentially fast to the unique discrete steady state $u_h \equiv 1$.

\section{Cross diffusion systems} \label{sec:7}

Another class of problems that initiated substantial research efforts in the area of entropy methods 
are cross diffusion systems
\begin{alignat*}{2}
\dt w &= \div (A(w) \nabla w), \qquad  && \text{in } \Omega, \ t>0,\\
    0 &= A(w) \dn w, \qquad && \text{on } \partial\Omega, \ t>0.    
\end{alignat*}
Here $w : \Omega \to \RR^n$ is vector valued and $\div (A(w) \nabla w)_i = \sum_j \sum_k  \partial_j (A(w)_{ik} \partial_j w_k)$ for some matrix valued function $A(w)$; 
the term $A(w) \dn w$ denotes the corresponding normal derivatives.
The evolution is equipped with an entropy functional 
$E(w)=\int_\Omega e(w) dx$ with entropy density $e(\cdot)$ that is assumed to be smooth and strictly convex. 

Following the arguments of \cite{BurgerDiFrancescoPietschmannSchlake10,Juengel15}, 
we first transform the system into \emph{entropy variables} 
\begin{align*}
u = u(w) := e'(w).
\end{align*}
Note that $e'(\cdot)$ can be assumed invertible, since $e(\cdot)$ is strictly convex. 
We can thus recover the physical fields from the entropy variables via
\begin{align*}
w = w(u) = (e')^{-1}(u).
\end{align*}
By substituting $w=w(u)$ into the cross-diffusion system stated above, 
we obtain the following equivalent system in entropy variables
\begin{alignat*}{2}
[e''(w(u))]^{-1} \dt u &= \div (B(u) \nabla u), \qquad && \text{in } \Omega, \ t>0, \\
                     0 &= B(u) \dn u, \qquad && \text{on } \partial\Omega, \ t>0, 
\end{alignat*}
with diffusion matrix $B(u) = A(w(u)) \cdot [e''(w(u))]^{-1}$. 
The basic assumption for the analysis of the cross diffusion system now is, that the entropy density $e(w)$ 
can be chosen such that $B(u) = A(w(u)) [e''(w(u))]^{-1}$ is symmetric and positive semi-definite. 

The natural choice of an energy for the system in entropy variables is 
\begin{align*}
\E(u) = E(w(u)).
\end{align*}
By elementary calculations, one can verify that
\begin{align*}
\langle \E'(u), v\rangle
&=\int_\Omega e'(w(u)) w'(u) v dx 
 = \int_\Omega u [e''(w(u))]^{-1} v dx 
 = \langle  [e''(w(u))]^{-1} w, v\rangle,
\end{align*}
where we used that the Hessian matrix $e''(w)$ is symmetric in the last step. 
We can thus decompose 
$\E'(u) = \Q(u) u$ with $\Q(u)$ and adjoint $\Q(u)^*$ formally defined by
\begin{align*}
\Q(u) v = [e''(w(u))]^{-1} v \qquad \text{and} \qquad \Q(u)^* v = [e''(w(u))]^{-1} v. 
\end{align*}
We again identified $\Q(u)$ and $\Q(u)^*$ with the multiplication operators defining them.
With $\A(u):=\div (B(u) \nabla u)$, 
the cross diffusion system in entropy variables 
can then be written in the abstract form $\Q(u)^* \dt u = \A(u)$ required for our framework. 
Under the above assumption that $B(u)$ is symmetric and positive semi-definite, we obtain 
\begin{align*}
\langle \A(u),u \rangle  = - \langle B(u) \nabla u, \nabla u\rangle := - \D(u) \le 0.
\end{align*}

As a particular example, let us consider the $2 \times 2$ system studied in \cite{Juengel15}, where
\begin{align*}
A(w) = \frac{1}{2+4w_1+w_2} \begin{pmatrix} 1+2 w_1 & w_1 \\ 2w_2 & 2+w_2 \end{pmatrix}. 
\end{align*}
This system models diffusion in a three component system with mass fractions $w_1$, $w_2$, and $w_3=1-w_1-w_2$.   
An appropriate entropy for the evolution is given by 
\begin{align*}
E(w) = \int_\Omega e(w) dx \qquad \text{with} \qquad e(w)=\sum_{i=1}^3 w_i (\log w_i -1). 
\end{align*}
By elementary computations, one can verify that 
\begin{align*}
\frac{d}{dt} E(t) =  -\int_\Omega 2 |\nabla \sqrt{w_1}|^2 + 4 |\nabla \sqrt{w_2}|^2 dx =: -D(w),
\end{align*}
which is crucial for establishing the global existence of solutions. 
The transformations between physical and entropy variables 
here read 
\begin{align*}
u_i &= \log\left(\frac{w_i}{1-w_1-w_2}\right)
\qquad \text{and} \qquad 
w_i = \frac{e^{u_i}}{1+e^{u_1}+e^{u_2}}.
\end{align*}
The back transformation to physical variables automatically yields $0 < w_i < 1$. 
The two matrices relevant for the system in entropy variables are further given by 
\begin{align*}
e''(w) = 
\begin{pmatrix} 
\frac{1}{w_1} + \frac{1}{1-w_1-w_2} &               \frac{1}{1-w_1-w_2} \\
                \frac{1}{1-w_1-w_2} & \frac{1}{w_2} + \frac{1}{1-w_1-w_2}  
\end{pmatrix}
\end{align*}
and
\begin{align*}
B(u(w)) = \frac{1}{(2+4w_1+w_2)} 
\begin{pmatrix}
w_1 (1+w_1-2w_1^2-w_1w_2) & -w_1w_2(2 w_1+w_2) \\ -w_1w_2 (2w_1+w_2) & w_2 (2-w_2-2w_1w_2-w_2^2)
\end{pmatrix}.
\end{align*}
The corresponding formulas for $e''(w(u))$ and $B(u)$ follow simply by inserting the expression for $w=w(u)$.
Both matrices are obviously symmetric and can be shown to be positive definite for arguments $0 < w_i(u) < 1$; see above.

For the numerical approximation of the cross-diffusion system in entropy variables, we can now simply apply a standard finite element approximation in space and a discontinuous Galerkin method in time.
By the results of Section~\ref{sec:2} and \ref{sec:3}, the corresponding numerical method inherits the underlying  dissipation structure automatically.
Another strategy for a structure preserving discretization based on mixed finite elements was proposed in \cite{BurgerCarilloWolfram10}.

\section*{Part III: Problems with energy conservation and dissipation}

While the previous examples were all concerned with diffusive partial differential equations, for which $\E(u)$ often has an interpretation as a physical entropy, we now turn 
to some typical applications that describe conservation or dissipation of energy. 

\section{Nonlinear electromagnetics} \label{sec:8}

The propagation of high-intensity electromagnetic fields through a non-dispersive absorbing medium is described by Maxwell's equations 
\begin{alignat*}{3}
\dt D = \curl H -  \sigma(E) E, \qquad \dt B = -\curl E,  \qquad \text{in } \Omega, \ t>0.
\end{alignat*}
Here $\sigma(E)$ denotes the conductivity of the medium, which may in general be field dependent. We assume that the electric and magnetic field intensities are coupled to the corresponding flux densities by constitutive equations of the form
\begin{alignat*}{2}
 D = d(E), \qquad B = b(H),
\end{alignat*}
which may again be nonlinear in the case of high field intensities. 
We further assume that $d,b:\RR^3 \to \RR^3$ are smooth functions with derivatives 
$d'(E),b'(H) \in \RR^{3 \times 3}$ being symmetric and positive definite, viz., the \emph{incremental permittivity} and \emph{permeability}.

A typical example for the constitutive equations is given by
\begin{align*}
d(E) = \eps_0 (\chi^{(1)} + \chi^{(3)} |E|^2) E, \qquad b(H) = \mu_0 H,  
\end{align*}
where $\eps_0,\mu_0$ denote the permittivity and permeability of vacuum, and the positive constants $\chi^{(1)},\chi^{(3)}$ describe the nonlinear dielectric response of a Kerr medium.

For ease of presentation, we assume in the sequel that $\Omega \subset \RR^3$ is bounded and that 
\begin{align*}
E \times n = 0, \qquad \text{on } \partial \Omega, \ t>0,
\end{align*}
i.e., the computational domain is enclosed in a perfectly conducting box. Other suitable boundary conditions could be treated with obvious modifications. 

In order to characterize the electromagnetic energy of the system, 
we introduce two scalar potentials, i.e., the electric and magnetic energy densities
\begin{align*}
\widehat d(E) = \int_0^E E \cdot d'(E) \cdot dE, \qquad \widehat b(H) = \int_0^H H \cdot b'(H) \cdot dH,
\end{align*}
which are to be understood as path integrals. 
The total energy content of an electromagnetic field distribution $(E,H)$ inside the domain $\Omega$ is then given by
\begin{align*}
\E(E,H) = \int_\Omega \widehat d(E)  + \widehat b(H) \; dx. 
\end{align*}
The derivative of the energy functional $\E(\cdot)$ can now be computed as 
\begin{align*}
\langle \E'(E,H), (\widetilde E, \widetilde H) \rangle 
= \int_\Omega E \cdot d'(E) \cdot \widetilde E + H \cdot b'(H) \cdot \widetilde H \; dx.
\end{align*}
Writing $u=(E,H)$ shows that the derivative can be decomposed as $\E'(u)=\Q(u) u$,
and the operators $\Q(u)$ and $\Q(u)^*$ can be identified with multiplication by the matrices
\begin{align*} 
\Q(E,H) = \begin{pmatrix} d'(E) & 0 \\ 0 & b'(H) \end{pmatrix} = \Q(E,H)^*.
\end{align*}
Using the constitutive relations, 
we can expand the time derivatives in Maxwell's equations as 
$\dt D=d'(E) \dt E$ and $\dt B=b'(H) \dt H$,
which leads to the equivalent system 
\begin{align*}
d'(E) \dt E = \curl H - \sigma E, \qquad b'(H) \dt H = -\curl E, \qquad  \text{in } \Omega, \ t>0.
\end{align*}
These equations already have the appropriate abstract form 
$\Q(u)^* \dt u = \A(u)$ of our framework with operator $\A(u)$ defined in a variational or corresponding weak form by 
\begin{align*}
\langle \A(E,H), (v,w) \rangle  
&= \langle \curl H,v\rangle  - \langle \sigma(E) E,v\rangle - \langle \curl E, w\rangle \\
&= \int_\Omega H \cdot \curl v - \sigma(E) E \cdot v - \curl E \cdot w \; dx.
\end{align*}
For the second identity, we used integration-by-parts and homogeneous boundary conditions $v \times n=0$ for the first test function.
Inserting $v=E$ and $w=H$ into the definition of $\A(\cdot)$ allows us to extract the dissipation functional
\begin{align*}
-\D(E,H) := \langle \A(E,H),(E,H) \rangle = -\int_\Omega \sigma(E) |E|^2 dx \le 0.
\end{align*}
From the abstract dissipation identity $\ddt \E(u) \le - \D(u)$ provided by Theorem~\ref{thm:1}, we can thus conclude that the energy of the electromagnetic system is conserved over time up to the part that is dissipated by conduction losses. 

A quick inspection of the above weak form of the operator $\A(\cdot)$ shows that the natural function spaces for the representation of the fields $E(t)$ and $H(t)$ here are given by $H_0(\curl,\Omega)$ and $L^2(\Omega)$. 
A Galerkin approximation of the weak formulation of the evolution problem is then possible by appropriate mixed finite elements \cite{BoffiBrezziFortin,Monk}. Together with a discontinuous Galerkin discretization in time, we obtain numerical approximation schemes that automatically inherit the underlying energy dissipation structure; this follows directly from the abstract results of Section~\ref{sec:3}.

\subsection*{Remark.}
Some very popular discretization schemes, viz., the finite-difference-time-domain method \cite{Yee} and the finite-integration-technique \cite{Weiland}, are based on a formulation in different variables, e.g.,  $E$ and $B$, and also on other time discretization strategies.  
It seems open or at least unclear, to which extent these methods are able to represent the underlying energy structure correctly on the discrete level.

\section{Gas dynamics} \label{sec:9}

The following example taken from \cite{Egger18} was actually our main motivation for developing the abstract framework 
presented in this paper. 
The isentropic flow of gas through a pipe of length $L$ is governed by balance laws of the form
\begin{alignat*}{2}
          \dt \rho + \dx q &= 0,            \qquad && 0<x<L, \ t>0, \\
\dt q + \dx (q^2/\rho + p) &= - q |q|/\rho, \qquad && 0<x<L, \ t>0,
\end{alignat*}
which describe the conservation of mass and the balance of momentum, respectively. The right hand side of the second equation models the friction at the pipe walls
and, for ease of notation, all irrelevant parameters were scaled here to one; 
we refer to   \cite{BrouwerGasserHerty11} for more information on the model and further references.
In order to close the system, we require that the pressure and density are related by an equation of state, e.g.,
\begin{align*}
p=p(\rho)=\rho^\gamma, \qquad \gamma>1.
\end{align*}
We further assume that the pipe is closed at the ends, which can be expressed as 
\begin{align*}
q(0)=q(L)=0, \qquad t>0. 
\end{align*}
The total free energy of the gas transport problem, consisting of a kinetic and an internal energy contribution, 
is then given by 
\begin{align*}
\E(\rho,q) = \int_0^L \frac{q^2}{2\rho} + P(\rho) \; dx, 
\end{align*}
where $P(\rho)=\rho \int_1^\rho p(r)/r^2 dr$ denotes the density of the internal energy.
Using the two balance laws above and the boundary conditions, one can show that 
\begin{align*}
\ddt \E(\rho(t)),q(t)) = -\int_0^L \frac{|q|^3}{\rho^2} dx \le 0,
\end{align*}
i.e., energy is conserved up to a part that is dissipated by friction at the pipe walls; a proof is presented below.
The derivative of the energy functional is here given by 
\begin{align*}
\langle \E'(\rho,q), (\tilde \rho,\tilde q)\rangle  
&= \int_0^L -\frac{q^2}{2\rho^2} \tilde \rho + P'(\rho) \tilde \rho + \frac{q}{\rho} \tilde q \; dx .
\end{align*}
A simple calculation allows to decompose the derivative as $\E'(\rho,q) = \Q(\rho,q) (\rho,q)$ with an operator $\Q(\rho,q)$ that can be identified with multiplication by the matrix 
\begin{align*}
\Q(\rho,q) = \begin{pmatrix} \frac{P'(\rho)}{{\rho}} & -\frac{q}{2\rho^2} \\ 0 & \frac{1}{\rho}\end{pmatrix}. 
\end{align*}
The adjoint operator $\Q(\rho,q)^*$ can then be identified with multiplication by the transposed matrix. 
By means of these operators, we can rewrite the above balance equations in the abstract form $\Q(u)^* \dt (u) = \A(u)$ required for our framework with $u=(\rho,q)$. The corresponding differential equations now read 
\begin{align*}
\frac{P'(\rho)}{\rho} \dt \rho &=  -\frac{P'(\rho)}{\rho} \dx q,  \\
\frac{1}{\rho} \dt q - \frac{q^2}{2\rho^2} \dt \rho &= -\dx\left( \frac{q^2}{2\rho^2} + P'(\rho)\right) - \frac{q}{2\rho^2} \dx q  - \frac{|q| q}{\rho^2},
\end{align*}
and they are again supposed to hold for all $0<x<L$ and $t>0$.

A weak formulation of this system can be obtained by 
testing the two equations with test functions $\eta$ and $w$, using integration-by-parts for the first term on the right hand side of the second equation, and imposing homogeneous boundary conditions for the test function $w$. The resulting variational principle reads 
\begin{alignat*}{2}
\left\langle \frac{P'(\rho)}{\rho} \dt \rho\,\eta \right\rangle 
&= - \left\langle \frac{P'(\rho)}{\rho} \dx q, \eta\right\rangle,  \\
\left\langle \frac{1}{\rho} \dt q - \frac{q^2}{2\rho^2} \dt \rho, w  \right\rangle &= 
\left\langle  \frac{q^2}{2\rho^2} + P'(\rho), \dx w \right\rangle 
- \left\langle \frac{q}{2\rho^2} \dx q + \frac{|q| q}{\rho^2},w \right\rangle,
\end{alignat*}
for all $\eta \in L^2(0,L)$, $w \in H^1_0(0,L)$, and all $t>0$. 
Note that the two solution components $\rho=\rho(t)$ and $q=q(t)$ depend on time, while the test functions $\eta$ and $w$ are independent of time. 
Simply testing this variational principle with $\eta=\rho(t)$ and $w=q(t)$ results in 
\begin{align*}
\ddt \E(\rho,q) 
&= \left\langle \frac{P'(\rho)}{\rho} \dt \rho, \rho\right\rangle + \left\langle \frac{1}{\rho} \dt q - \frac{q^2}{2\rho^2} \dt \rho, q \right\rangle \\
&= -\left\langle \frac{|q| q}{\rho^2},q\right\rangle
 = -\int_0^L |q|^2/\rho^2 dx =: -\D(\rho,q).
\end{align*}
This is exactly the energy dissipation identity announced above; see also Theorem~\ref{thm:1}. 

The advantage of this, somewhat non-conventional, formulation of the gas transport problem is, that a systematic discretization of the corresponding weak formulation is now possible by Galerkin approximation. 
As illustrated in \cite{Egger18}, a space discretization by piecewise linear finite elements for $\rho$ and continuous piecewise linear finite elements for $q$ leads to a semi-discretization that inherits the underlying dissipation structure. 
A subsequent time discretization by the implicit Euler method 
yields a fully discrete approximation that obeys a corresponding dissipation inequality. 
As shown in \cite{Egger18}, the discretization scheme can be extended naturally from a single pipe to pipeline networks and, although no particular upwind technique was employed, the scheme also seems to capture the correct behavior in the presence of shocks.

\section{Hamiltonian systems} \label{sec:10}

As a last example, we consider Hamiltonian or gradient systems of the form 
\begin{align*}
\dot x &= (J(x) - R(x)) \nabla_x H(x).
\end{align*}
Here $H : X \to \RR$ is a given energy functional, $X$ denotes an appropriate state space, and $\nabla_x H(x)$ denotes the Riesz-representation of the derivative functional $H'(x)$. 
The operators $J(x),R(x) : X \to X'$ are assumed to be anti-symmetric and positive semi-definite, respectively; see e.g. \cite{ChaturantabutBeattieGugercin16,SchJ14} for details. 
By these assumptions, we have
\begin{align*}
\langle J(x) y, y\rangle =0 \qquad \text{and} \qquad \langle R(x) y,y\rangle \ge 0 \qquad \text{for all } x,y \in X.
\end{align*}
The energy balance of the dynamical system can then be derived as follows:
\begin{align*} 
 \frac{d}{dt} H(x(t)) 
&= \langle \nabla_x H(x(t)), \dot x(t)\rangle \\
&= \langle \nabla_x H(x(t)), J(x) \nabla_x H(x(t))\rangle - \langle \nabla_x H(x(t)), R(x(t)) \nabla_x H(x(t))\rangle \\
&= - \langle \nabla_x H(x(t)), R(x(t)) \nabla_x H(x(t))\rangle =: -D(x(t)) \le 0.
\end{align*}
Under the above assumptions on the operators $J(x)$ and $R(x)$, the energy $H(x(t))$ of the system is thus monotonically decreasing along smooth solution trajectories.

In a similar manner as in Section~\ref{sec:7}, we now introduce the transformation to entropy variables and the corresponding back transformation according to
\begin{align*}
u = u(x) :=\nabla_x H(x) 
\qquad \text{and} \qquad 
x=x(u)=(\nabla_x H)^{-1}(u).
\end{align*}
We tacitly assumed here that the function $\nabla_x H(\cdot)$ is invertible. 
Similarly as in Section~\ref{sec:7}, 
we can then equivalently rewrite the evolution equation in entropy variables as
\begin{align*}
[\nabla_{xx} H(x(u))]^{-1} \dt u &= [J(x(u)) - R(x(u))] \; u.
\end{align*}
The energy of the system in entropy variables is simply given by $\E(u)=H(x(u))$.\\
By some elementary calculations, one can verify that 
\begin{align*}
\nabla_w \E(u) 
= \nabla_u x(u) \nabla_x H(x(u)) 
= [\nabla_{xx} H(x(u))]^{-1} u.  
\end{align*}
This shows that the above dynamical system written in entropy variables has exactly the form $\Q(u)^*\dt u = \A(u)$ required for our framework with operators 
\begin{align*}
\Q(u)^* v =  [\nabla_{xx} H(x(u))]^{-1} v 
\qquad \text{and} \qquad 
\A(u) = [J(x(u)) - R(x(u))] \, u.
\end{align*}
By the above assumptions on $J(x)$ and $R(x)$, the dissipation functional satisfies
\begin{align*}
-\D(u):=\langle \A(u), u\rangle = -\langle R(x(u)) u, u \rangle \le 0.
\end{align*}
From Theorem~\ref{thm:1}, we thus obtain the dissipation identity $\ddt \E(u) = - \D(u)$, which is of course equivalent to the identity $\ddt H(x) = - D(x)$ stated above. 

From our considerations in Section~\ref{sec:3}, we can further deduce that a simple Galerkin approximation in space of the system in entropy variable and a time discretization by a discontinuous Galerkin method will automatically lead to numerical approximations that inherit the dissipative nature of the underlying Hamiltonian or gradient system. 

\subsection*{Remark.}
Our framework also provides a systematic strategy for the structure preserving model reduction of Hamiltonian or more general gradient systems; let us refer to \cite{BenMS05,ChaturantabutBeattieGugercin16,Schilders08} for an introduction to the field. 
Following our abstract framework, the Hamiltonian or gradient structure can automatically be preserved in the model reduction process, if the system is first rewritten in entropy variables and then a Galerkin projection is used for the construction of the reduced model. A discontinuous Galerkin approximation in time 
allows to obtain even fully discrete approximate models which automatically preserve the underlying Hamiltonian or gradient structure.

\section*{Discussion}

In this paper, we proposed an general abstract framework for the formulation and systematic discretization of evolution problems that are governed by energy dissipation or entropy production.
The basic step in our approach was to rewrite the problem 
in a particular form that complies with the underlying dissipation structure. 
A structure-preserving numerical approximation could then be achieved by Galerkin approximation in space and discontinuous Galerkin methods in time. 
As we demonstrated, the proposed framework is applicable to a wide range of applications, including diffusive partial differential equations and Hamiltonian or more general gradient systems. 

While the general strategy for the design of structure-preserving discretization schemes seems formally applicable to almost any dissipative problem, 
the numerical analysis of the resulting schemes, apart from their dissipation behavior, remains problem dependent and still has to be done case by case. 
We strongly believe, however, that a systematic numerical analysis might be possible for certain classes of applications under rather general assumptions on the main ingredients, e.g. the energy and dissipation functionals and the function spaces used for the formulation.
We hope that this article will initiate further research in this direction.

\section*{Acknowledgments}
The author would like to thank the German Research Foundation (DFG) for financial support through the grants Eg-331/1-1, IRTG~1529, TRR~146 and TRR~154, and through the ``Excellence Initiative'' of the German Federal and State Governments via the Graduate School of Computational Engineering GSC~233 at TU~Darmstadt.


\end{document}